\documentclass[12pt]{amsart}
\usepackage{latexsym}
\usepackage{amstext}
\usepackage{amssymb}
\usepackage{amsfonts}
\usepackage{amsthm}
\usepackage{amsopn}
\usepackage{amsbsy}
\usepackage{layout}
\usepackage{color}
\usepackage{amsmath}
\usepackage{graphicx}
\usepackage{bm}
\usepackage{yfonts}
\usepackage{orcidlink} 
\newtheorem{theorem}{Theorem}[section]

\newtheorem{proposition}[theorem]{Proposition}
\newtheorem{question}[theorem]{Question}

\theoremstyle{definition}
\newtheorem{definition}[theorem]{Definition}
\newtheorem{remark}[theorem]{Remark}

\numberwithin{equation}{section}

\begin{document}

\vspace{0.5in}

\newcommand{\abs}[1]{\lvert#1\rvert}
\def\norm#1{\left\Vert#1\right\Vert}
\def\Q {{\Bbb Q}}
\def\I {{\Bbb I}}
\def\C {{\Bbb C}}
\def\N{{\Bbb N}}
\def\R{{\mathbb R}}
\def\di{{\mathrm{di}}}
\def\Z {{\Bbb Z}}
\def\U{{\Bbb U}}
\def\F{{\mathrm{E}}}
\def\Un{{\mathcal{U}}}
\def\Is{{\mathrm{Is}}\,}
\def\Aut{{\mathrm {Aut}}\,}
\def\supp{{\mathrm {supp}}\,}
\def\Homeo{{\mathrm{Homeo}}\,}
\def\gr{{\underline{\Box}}}
\def\diam{{\mathrm{diam}}\,}
\def\d{{\mathrm{dist}}}
\def\H{{\mathcal H}}
\def\me{{\mathrm{me}}}
\def\mO{{\overline{o}}}

\def\a{\alpha}
\def\d{\delta}
\def\D{\Delta}
\def\g{\gamma}
\def\s{\sigma}
\def\Si{\Sigma}
\def\implies{\Rightarrow}
\def\o{\omega}
\def\O{\Omega}
\def\G{\Gamma}

\def\sB{{\mathcal B}}
\def\sC{{\mathcal C}}
\def\sE{{\mathcal E}}
\def\sF{{\mathcal F}}
\def\sG{{\mathcal G}}
\def\sH{{\mathcal H}}
\def\sJ{{\mathcal J}}
\def\sK{{\mathcal K}}
\def\sL{{\mathcal L}}
\def\sM{{\mathcal M}}
\def\sN{{\mathcal N}}
\def\sO{{\mathcal O}}
\def\sP{{\mathcal P}}
\def\sR{{\mathcal R}}
\def\sS{{\mathcal S}}
\def\sT{{\mathcal T}}
\def\sU{{\mathcal U}}
\def\sV{{\mathcal V}}

\def\sbs{\subset}
\def\rar{\rightarrow}
\def\e{\epsilon}

\def\ti{\times}
\def\obr{^{-1}}
\def\stm{\setminus}
\def\newline{\hfill\break}

\def\Exp{{\mathrm{Exp}}\,}
\def\Iso{{\mathrm{Iso}}\,}
\def\Sym{{\mathrm{Sym}}\,}


\title[On the product of Weak Asplund locally convex spaces]
{On the product of Weak Asplund locally convex spaces}
\author[Jerzy K\c{a}kol and Arkady Leiderman]
{
Jerzy K\c{a}kol
\orcidlink{0000-0002-8311-2117}
and Arkady Leiderman 
\orcidlink{0000-0002-2257-1635}
}
\address{Faculty of Mathematics and Informatics, A. Mickiewicz University,
61-614 Pozna\'{n}, Poland and Institute of Mathematics Czech Academy of Sciences, Prague, Czech Republic}
\email{kakol@amu.edu.pl}

\address{Department of Mathematics, Ben-Gurion University of the Negev, Beer Sheva, P.O.B. 653, Israel}
\email{arkady@math.bgu.ac.il}

\begin{abstract}
For locally convex spaces, we systematize several known equivalent definitions of Fr\'echet (G\^ ateaux) Differentiability Spaces and Asplund (Weak Asplund) Spaces.

As an application, we extend the classical Mazur's theorem as follows:
Let $E$ be a separable Baire locally convex space and let $Y$ be the product $\prod_{\alpha\in A} E_{\alpha}$ 
of any family of separable Fr\'echet spaces; then the product $E \times Y$ is Weak Asplund.
Also, we prove that the product $Y$ of any family of Banach spaces $(E_{\alpha})$ is an Asplund locally convex space if and only if each 
$E_{\alpha}$ is Asplund.

Analogues of both results are valid under the same assumptions, if $Y$ is the $\Sigma$-product of any family $(E_{\alpha})$.
\end{abstract}

\maketitle
\section{Equivalence of definitions}\label{Equivalence}
\bigskip

All topological spaces $X$ are assumed to be Tychonoff and all topological vector spaces are Hausdorff.
All vector spaces are considered over the field of real numbers $\mathbb R$.
The abbreviation lcs means a locally convex space. 
A completely metrizable lcs is called a Fr\'echet lcs.
A topological space $X$ is said to be {\it Baire} space if the intersection of every countable sequence of dense open subsets in $X$ is dense.
By a {\it bounded set} in a topological vector space we understand any set which is absorbed by every $0$-neighbourhood.

There are many senses in which a map between topological linear spaces can be considered differentiable 
(see \cite{AS} for details and history), the choice often depends on the applications under consideration.
The authors of notes \cite{Keller}, \cite{Yamamuro} have suggested and explored several approaches to overcome special difficulties that arise in non-normed cases.

As R. R. Phelps indicated in his thoroughly written book \cite{Phelps}, 
the first infinite-dimensional result about differentiability properties of convex functions was obtained in 1933 by S. Mazur \cite{Mazur}:
A continuous convex function $f:D \to \R$ defined on an open convex subset $D$ of a separable Banach space $E$, is G\^ ateaux differentiable on a dense $G_{\delta}$ subset of $D$.

In 1968, E. Asplund \cite{Asplund} extended Mazur's theorem in two ways. Namely, he proved the same statement for a more general class of Banach spaces;
and studied a more restricted class of Banach spaces (now called Asplund spaces) in which a 
stronger conclusion of Fr\'echet differentiability holds.

In the formulations of Definitions \ref{def1} - \ref{def4} below $D$ denotes a nonempty open convex subset of a lcs $E$.
\bigskip
\begin{definition}\label{def1}
A convex function $f: D \to \R$ is said to be Fr\'echet differentiable at a point $x_0\in D$
if there exists a continuous linear functional $f'(x_0)\in E'$ such that for any bounded subset $B\subset E$ and $\epsilon > 0$, there is $\delta > 0$ such that
$$\sup_{x\in B}\left|\frac{f(x_0+tx)-f(x_0)}{t} - \left\langle f'(x_0), x \right\rangle\right|<\epsilon$$
whenever $0< |t| <\delta$.
\end{definition}

\begin{definition}\label{def2} 
A convex function $f: D \to \R$ is said to be G\^ ateaux differentiable at a point $x_0\in D$
if there exists a continuous linear functional $f'(x_0)\in E'$ such that for any $x\in E$ and $\epsilon > 0$, there is $\delta > 0$ such that
$$\left|\frac{f(x_0+tx)-f(x_0)}{t} - \left\langle f'(x_0), x \right\rangle\right|<\epsilon$$
whenever $0< |t| <\delta$.
\end{definition}

Similarly to Banach spaces, locally convex spaces can be classified according
to the differentiability properties of the specified class of real valued convex functions. In the paper we follow the abbreviations suggested in \cite{Sharp}.

\begin{definition}\label{def3}
A lcs $E$ is called Asplund (Weak Asplund) if every continuous convex function $f: D \to \R$
is Fr\'echet (G\^ ateaux, respectively) differentiable on a dense $G_{\delta}$ subset of $D$.
Asplund (Weak Asplund) spaces are abbreviated by ASP (WASP, respectively). 
\end{definition}

\begin{definition}\label{def4}
A lcs $E$ is called Fr\'echet (G\^ ateaux) Differentiability Space if every continuous convex function $f: D \to \R$
is Fr\'echet (G\^ ateaux, respectively) differentiable on a dense subset of $D$.
Fr\'echet (G\^ ateaux) differentiable spaces are abbreviated by FDS (GDS, respectively). 
\end{definition}

The classification of locally spaces according to the dense or generic differentiability
of convex functions continues work of E. Asplund \cite{Asplund}, D. G. Larman and R. R. Phelps \cite{Larman} and I. Namioka and R. R. Phelps \cite{Namioka}.
It is known that for  Banach spaces the points of  Fr\'echet differentiability of any continuous convex function form a $G_{\delta}$ set in its domain, 
so for Banach spaces ASP and FDS are equivalent. 
M. Talagrand \cite{Tal1} constructed a 1-Lipschitz function on a Banach space $C(K)$ whose set of points of G\^ ateaux differentiability is dense but of the first category.
 Moreover, M. \v Coban and P. Kenderov \cite{Coban} have given examples to show that even when the set
of G\^ ateaux differentiability points of the $\sup$-norm of a Banach space $C(K)$ is dense, it need not contain a $G_{\delta}$ set.
In particular, this phenomenon occurs for the {\it double arrows} compact space $K$.
There is an example of a Banach space that is GDS but not Weak Asplund \cite{Moors}.

A systematic research of locally convex spaces that are ASP, WASP, FDS, and GDS was originated in 1990 by B. Sharp \cite{Sharp}, and then it was developed further in joint papers \cite{Eyland_Sharp1}, \cite{Eyland_Sharp2},
 \cite{Eyland_Sharp3}.
Among other results, B. Sharp \cite{Sharp} obtained a generalization of Mazur's Theorem.
\begin{theorem}\label{theorem_sharp}\cite[Theorem 2.1]{Sharp}
 All separable Baire topological linear spaces (not necessarily lcs) are WASP.
\end{theorem}

Unfortunately, several authors much later rediscovered this result without any proper citation of the original publication \cite{Sharp};
the list of such papers includes \cite{Corbacho}, \cite{Lee}, \cite{Zheng}. 
Regrettably, it appears that other important works \cite{Eyland_Sharp1}, \cite{Eyland_Sharp2}, \cite{Eyland_Sharp3} have also not received due attention in the later publications.

Let us emphasize that in Definitions \ref{def3} and \ref{def4} one consider a collection of arbitrary continuous convex functions $f: D \to \R$,
where $D\subset E$ is any nonempty open convex set. A natural question arises whether we can take a subcollection of more specific functions $f$ and obtain the same classes of lcs $E$.
The following approach in this direction was developed in \cite{Eyland_Sharp2}: Consider only $f: E \to \R$ which are the {\it Minkowski gauges} of open disks,
equivalently, consider only arbitrary {\it seminorms} defined on lcs $E$.
For our purposes, the principal results obtained in \cite{Eyland_Sharp2} can be reformulated as follows.

\begin{theorem}\label{theorem_eq1} \cite{Eyland_Sharp2} Let $E$ be a lcs.\hfill
\begin{enumerate}
\item[{\rm (1)}] $E$ is GDS if and only if every seminorm on $E$ is G\^ ateaux differentiable on a dense subset of $E$.
\item[{\rm (2)}] $E$ is FDS if and only if every  seminorm on $E$ is Fr\'echet differentiable on a dense subset of $E$. 
\item[{\rm (3)}] Assume that $E$ is Baire. Then $E$ is WASP if and only if every seminorm on $E \times \R$ is G\^ ateaux differentiable on a dense $G_{\delta}$ subset of $E \times \R$.
\item[{\rm (4)}] Assume that $E$ is Baire. Then $E$ is ASP if and only if every seminorm on $E \times \R$ is Fr\'echet differentiable on a dense $G_{\delta}$ subset of $E \times \R$.
\end{enumerate}
\end{theorem}

As another natural option, one can take a subcollection of those continuous convex functions $f$ which domains coincide with the whole $E$.
Since every seminorm is a continuous convex function, Theorem \ref{theorem_eq1} immediately implies

\begin{proposition}\label{prop_eq2} Let $E$ be a lcs.\hfill
\begin{enumerate}
\item[{\rm (1)}] $E$ is GDS if and only if every continuous convex function $f: E \to \R$ is G\^ ateaux differentiable on a dense subset of $E$.
\item[{\rm (2)}] $E$ is FDS if and only if every continuous convex function $f: E \to \R$ is Fr\'echet differentiable on a dense subset of $E$.
\end{enumerate}
\end{proposition}

Analogous equivalencies for WASP and ASP provided $E$ is a Banach spaces are very well known. However, the published proofs of these equivalencies
are very indirect and use a heavy machinery of the Banach spaces tools (see \cite[Theorem 8.26]{Fabian1}, \cite[Section I.5]{DGZ}).
It is not clear whether these arguments can be easily adapted to the case of non-Banach spaces.
Nevertheless, the statement is true for any Baire lcs. It is sufficient to apply another result obtained in \cite{Eyland_Sharp2}. 

We define an {\it asymptotic seminorm} on a lcs $E$ (see \cite{Eyland_Sharp2}) to be a continuous function $f: E \to \R$ for which there is a seminorm $g$ on $E$ satisfying:
\begin{enumerate}
\item[{\rm (a)}] If $x_{\alpha}\rightarrow x$ in $E$ and $\lambda \rightarrow \infty$ then $\frac{1}{\lambda} f(\lambda x_{\alpha}) \rightarrow g(x)$;
\item[{\rm (b)}] For all $x \in E$, $f(x)=f(-x)$ and $f(x)\geq g(x)$. 
\end{enumerate}

An asymptotic seminorm is convex and every seminorm is an asymptotic seminorm.
\begin{theorem}\label{theorem_eq3} \cite{Eyland_Sharp2} Let $E$ be a Baire lcs.\hfill
\begin{enumerate}
\item[{\rm (1)}] $E$ is WASP if and only if every asymptotic seminorm on $E$ is G\^ ateaux differentiable on a dense $G_{\delta}$ subset of $E$.
\item[{\rm (2)}] $E$ is ASP if and only if every asymptotic seminorm on $E$ is Fr\'echet differentiable on a dense $G_{\delta}$ subset of $E$.
\end{enumerate}
\end{theorem} 

Theorem \ref{theorem_eq3} immediately implies

\begin{proposition}\label{prop_eq4} Let $E$ be a Baire lcs.\hfill
\begin{enumerate}
\item[{\rm (1)}] $E$ is WASP if and only if every continuous convex function $f: E \to \R$ is G\^ ateaux differentiable on a dense $G_{\delta}$ subset of $E$.
\item[{\rm (2)}] $E$ is ASP if and only if every continuous convex function $f: E \to \R$ is Fr\'echet differentiable on a dense $G_{\delta}$ subset of $E$.
\end{enumerate}
\end{proposition}

\section{Application to products}\label{Application} 
\bigskip

Theorem \ref{theorem_eq1} prompts a question whether the above defined classes of lcs are closed under products.
If $\{E_i: i =1, \dots, n\}$ is a finite family of Banach Asplund spaces, then so is the product $\prod_{i=1}^n E_i$.
Does $E$ is a Banach WASP imply that $E \times \R$ also is WASP,
is a long standing open question. For general lcs the answer is not known either.

M. Fabian showed that if $E$ is a Banach GDS, then so is the product $E \times \R$ (see \cite[Proposition 6.5]{Phelps}).
Later this result was extended to the products $E \times Y$, where is $Y$ any separable Banach space \cite{Cheng_Fabian}. 
Note that R. Eyland and B. Sharp \cite[3.2]{Eyland_Sharp2} proved for any lcs $E$: if $E$ is GDS (FDS), then so is $E\times \R$, respectively. 
A further generalization for the product of lcs was obtained in \cite{Shen_Cheng}: let a lcs $E$ be GDS and let $Y$ be the product $\prod_{\alpha\in A} E_{\alpha}$ 
of any family of separable Fr\'echet spaces; then the product $E \times Y$ also is GDS.

To the best of our knowledge, the questions about the preservation of WASP and ASP for arbitrary products of locally convex spaces have been not considered in the literature before.
First, we observe how to extend Theorem \ref{theorem_sharp}. Note also that the product of two normed Baire spaces need not to be Baire \cite{Valdivia}.

\begin{proposition}\label{prop_1}
Let $E$ be a separable Baire topological linear space. Then for every separable Fr\'echet space $Y$ 
the product $E \times Y$ is WASP.
\end{proposition}

\begin{proof} Clearly, $E \times Y$ is separable. Moreover, the product $E \times Y$ is Baire.
This claim follows from the remarkable result of W. Moors \cite{Moors_Baire}:
The product of a Baire space with a metrizable hereditarily Baire space is again a Baire space. Evidently, any completely metrizable space $Y$ is
hereditarily Baire, i.e. every closed subspace of $Y$ is Baire. Thus, $E \times Y$ is WASP by Theorem \ref{theorem_sharp}.
\end{proof} 

Next we prove a general statement which is one of the main results of our paper.
Undoubtedly, our arguments are inspired by the proof of \cite[Theorem 5.5]{Sharp} stating that any lcs with a weak topology is ASP.

\begin{theorem}\label{theorem_1}
Let $E$ be a separable Baire lcs and let $Y$ be the product $\prod_{\alpha\in A} E_{\alpha}$ 
of any family of separable Fr\'echet spaces. 
Then the product $E \times Y$ is WASP.
\end{theorem}
\begin{proof} We shall use two apparently folklore facts.
\\
\\
{\bf Fact 1.} {\it Let $f: X \to \R$ be a convex function defined on a linear space $X$. If $f$ is bounded then it is constant.\\
Proof.} If not, without loss of generality we may assume that $f(0)=0$ and there is $x\in X$ such that $f(x) > 0$. Then for every $\lambda > 1$ we have
$$f(x) = f(\frac{\lambda-1}{\lambda}\cdot 0+\frac{1}{\lambda}\cdot \lambda x) \leq \frac{1}{\lambda}f(\lambda x).$$
It follows that $f(\lambda x) \geq \lambda f(x)$, which contradicts to the boundedness of $f$. 
\\
\\
{\bf Fact 2.} (see \cite[Proposition 1.3.2]{Fabian})\hfill

 {\it Let $Z$ be a Baire topological space and let $\sU$ be an open cover of $Z$. Assume that $\Omega$ is a subset of $Z$ such that for any $U \in \sU$
the intersection $U\cap \Omega$ contains a dense $G_{\delta}$ subset in $U$. Then the whole set $\Omega$ contains a dense $G_{\delta}$ subset in $Z$.
}
\\
\\
In order to prove Theorem \ref{theorem_1} we fix a continuous convex function $f: E \times Y \to \R$.
We may assume that the set of indices $A$ is infinite, otherwise $Y$ is a separable Fr\'echet space, and we apply Proposition \ref{prop_1}. 
Denote the set of G\^ ateaux differentiability points of $f$ in $Z = E \times Y$ by $\Omega$.
According to Remark at the end of paper \cite{Moors_Baire}, the space $Z$ is Baire. 
Pick any point $P_0 \in Z$.
We shall find an open neighborhood $U$ of $P_0$ such that the intersection $U\cap \Omega$ contains a dense $G_{\delta}$ subset in $U$.

Since $f$ is continuous, it is bounded on some basic open convex neighborhood $U \subset Z$ of $P_0$.
Denote by $U = G \times \prod_{\alpha\in I} G_{\alpha} \times \prod_{\alpha\in A\setminus I} E_{\alpha}$,
where $I\subset A$ is finite, $G \subset E$ and $G_{\alpha} \subset E_{\alpha}$ are open convex neighborhoods.
Also, denote by $\pi$ the natural projection of $Y$ onto $Y_I = \prod_{\alpha\in I} E_{\alpha}$ and let $Z_I = E \times Y_I$.

Now we define a continuous convex function $h: G \times \prod_{\alpha\in I} G_{\alpha} \to \R$ by the rule: $h(x, \pi(y)) = f(x, y)$,
whenever $(x, y) \in U$.
Function $h$ is correctly defined because for every fixed $(x, a)\in G \times \prod_{\alpha\in I} G_{\alpha}$ 
a continuous convex function $\varphi(b) = f(x, (a, b)): \prod_{\alpha\in A\setminus I} E_{\alpha}\to \R$
is constant, by Fact 1.

According to Proposition \ref{prop_1}, $E \times Y_I$ is WASP. Therefore, function $h$ has a dense $G_{\delta}$ set $V$ 
of G\^ ateaux differentiability points in its open convex domain, and then function $f$ has a dense $G_{\delta}$ set $W = V \times \prod_{\alpha\in A\setminus I} E_{\alpha}$ 
of G\^ ateaux differentiability points in $U$. Indeed, let $z_0 = (x_0, y_0) \in W$.
For any $z=(x, y) \in  Z$ there is $\lambda > 0$ such that for all $|t| < \lambda$ and each $h' \in (Z_I)'$ we have
$$\left|\frac{h((x_0, \pi(y_0))+t (x, \pi(y)))- h(x_0, \pi(y_0))}{t} - \left\langle h', (x_0, \pi(y_0)) \right\rangle\right|=$$
$$=\left|\frac{f((x_0, y_0)+t (x, y))- f(x_0, y_0)}{t} - \left\langle f', (x_0, y_0) \right\rangle\right|,$$
	where $f' \in Z'$ is defined by $\left\langle f', (x, y) \right\rangle = \left\langle h', (x, \pi(y))\right\rangle$.
	This means that function $f$ is G\^ ateaux differentiable at $(x_0, y_0)$, by Definition \ref{def2}, because $h$ is G\^ ateaux differentiable at $(x_0, \pi(y_0))$. 

 Applying Fact 2 we conclude that the whole set $\Omega$ contains a dense $G_{\delta}$ subset in $Z$, and the proof is complete.
\end{proof}


\begin{theorem}\label{theorem_2}
Let $\{E_{\alpha}: \alpha \in A\}$ be any family of Banach spaces. 
Then the product $Y=\prod_{\alpha\in A} E_{\alpha}$ is an ASP lcs if and only if each $E_{\alpha}$ is ASP.
\end{theorem}
\begin{proof} 
First, any product of Banach spaces is a Baire topological space. Second, any finite product of Asplund Banach spaces is Asplund. Then we literally repeat the proof of Theorem \ref{theorem_2}.
Again, any continuous convex function defined on the product $Z$
factorizes thru a finite subproduct $\prod_{\alpha\in I} E_{\alpha}$. The only difference is the following.
At the end, in order to show Fr\'echet differentiability we need to add a trivial remark that for every bounded set in the product
its image under the projection to a subproduct is bounded.

In the opposite direction, let $f: E_{\beta} \to \R$ be a continuous convex function, where $\beta \in A$ is fixed.
Denote the projection of the product $Y$ onto $E_{\beta}$ by $\pi$. Observe that the projection $\pi$ is an open continuous and bound covering mapping,
i.e. for every bounded subset $B$ of $E_{\beta}$ there is a bounded subset $C$ of $Z$ such that $B \subset \pi(C)$.
Indeed, it is enough to include to $C$ all points which coordinates are zero for all indices $\alpha \neq \beta$ and identical to the points from $B$ on the place $\beta$.
Take a composition $f \circ \pi$. It is a continuous convex function defined on the Asplund lcs $Z$.
Hence, function $f \circ \pi$ has a dense set $\Omega \subset Z$ of Fr\'echet differentiability points. Evidently, $\pi(\Omega)$ is dense in $E_{\beta}$.
Notice that since $\pi$ is continuous and bound covering mapping, $f$ is Fr\'echet differentiable at $\pi(x)$ if and only if $f \circ \pi$  is Fr\'echet differentiable at $x$
 (see \cite{Eyland_Sharp1}). So, $f$ is Fr\'echet differentiable at each point of $\pi(\Omega)$. Finally, the Banach space $E_{\beta}$ is FDS and we are done since
for Banach spaces ASP and FDS are equivalent.
\end{proof}

More information about the behavior of ASP and WASP under linear continuous onto mappings of lcs can be found in \cite[Section 1]{Eyland_Sharp1}.

\begin{theorem}\label{theorem_map}\cite[Theorem 1.3]{Eyland_Sharp1}
Let $X$ be a Fr\'echet space, $Y$ a regular linear space, and let
$T: X  \to Y$ be continuous, linear, open and onto.
\begin{enumerate}
\item[{\rm (1)}] If $X$ is WASP, then so is $Y$.
\item[{\rm (2)}] Suppose that $T$ is also bound covering. If $X$ is ASP, then so is $Y$.
\end{enumerate}
\end{theorem}

Recall that the $\Sigma$-product of topological spaces $\{X_{\alpha}: \alpha \in A\}$, with
a base point $b = (b_{\alpha}) \in \prod_{\alpha\in A} X_{\alpha}$,
 is the subspace of the product $\prod_{\alpha\in A} X_{\alpha}$
consisting of points $x = (x_{\alpha})$ with $x_{\alpha}= b_{\alpha}$ for all but countably many $\alpha$.
Naturally, $\Sigma$-product is considered with a topology induced from the whole product.
It is easily seen that if $b_{\alpha}$ is the zero element of a lcs $E_{\alpha}$, then
the $\Sigma$-product of any family $\{E_{\alpha}: \alpha \in A\}$ is a locally convex space. 

It has been shown in \cite{CP} that the $\Sigma$-product of any family of completely metrizable spaces is a hereditarily Baire space.
Moreover, using the strategy of the main results of \cite{CP}, \cite{Moors_Baire} 
one can show that the product $Z = E \times Y$ is Baire, for every Baire space $E$ and $\Sigma$-product $Y$ of a family of completely metrizable spaces.

Therefore, the following counterparts of Theorems \ref{theorem_1} and \ref{theorem_2} are valid, with the same proofs, respectively.

\begin{proposition}\label{P_1}
Let $E$ be a separable Baire lcs and let $Y$ be the $\Sigma$-product 
of any family of separable Fr\'echet spaces. 
Then the product $E \times Y$ is WASP.
\end{proposition}

\begin{proposition}\label{P_2}
Let $\{E_{\alpha}: \alpha \in A\}$ be any family of Banach spaces. 
Then the $\Sigma$-product $Y$ of $\{E_{\alpha}: \alpha \in A\}$ is an ASP lcs if and only if each $E_{\alpha}$ is ASP.
\end{proposition}

\section{Open questions and remarks}\label{Questions} 
\bigskip

We lack the counterexamples to the very basic questions.

\begin{question}\label{problem_product1} Let $E$ be a lcs.
If $E$ is Asplund, is then $E\times \R$ Asplund?
\end{question}

If $E$ additionally is assumed to be Baire, Question \ref{problem_product1} appears as \cite[Conjecture 4.0]{Eyland_Sharp2}.

\begin{question}\label{problem_product2} Let $E$ be a lcs.
If $E\times \R$ is Asplund (Weak Asplund), is then $E$ Asplund (Weak Asplund, respectively)?
\end{question}

If $E$ additionally is assumed to be Baire, Question \ref{problem_product2} has been answered positively \cite{Eyland_Sharp2}.

\begin{question}\label{problem_product3} Let $E$ and $F$ be lcs.
  If $E\times F$ is Asplund (Weak Asplund), is then $E$ Asplund (Weak Asplund, respectively)? 
\end{question}

The answer to Question \ref{problem_product2} is affirmative, if additionally $E$ and $F$ are assumed to be Fr\'echet spaces. It follows from
Theorem \ref{theorem_map}.

\begin{remark}\label{remark} We should warn the reader that paper \cite{Lee} contains several results which proofs are not justified.
Let us provide some details.

 One of the claims of \cite[Theorem 5.5]{Lee} is the following: Let $E$ be a Baire lcs, then $E$ is ASP if and only if 
every seminorm on $E$ is Fr\'echet differentiable on a dense $G_{\delta}$ subset of $E$.
Such a statement would resolve immediately mentioned above \cite[Conjecture 4.0]{Eyland_Sharp2}.

However, a very short and indirect proof in \cite{Lee} is incomplete. The only correctly proved relevant result is Theorem
\ref{theorem_eq3}: A Baire lcs $E$ is ASP if and only if every {\it asymptotic seminorm} on $E$ is Fr\'echet differentiable on a dense $G_{\delta}$ subset of $E$.

Furthermore, \cite[Theorem 7.1]{Lee} claims a very impressive generalization of the Asplund's theorem, namely:
If the strong dual $E^{*}$ of a Baire lcs $E$ is separable, then $E$ is ASP.

We do not have a counterexample to this statement, but the proof provided in \cite{Lee} is incorrect.
First, the authors make an essential use of a problematic \cite[Theorem 5.5]{Lee}.
Second, let us comment the following argument of \cite{Lee}. 
Let $E$ be a lcs and let $E^{*}$ be endowed with the strong dual topology.
Take any continuous seminorm $p$ on $E$ and denote $E_0 =\{x \in E: p(x)=0\}$. Consider the quotient space $E_p=E/ E_0$ endowed with the following norm $\bar{p}(x + E_0)= p(x)$.
So, we have a normed space  $E_p$ (not necessarily complete).  Clearly, the natural map $Q: E\to E_p$ is a continuous linear surjection but not necessarily open. 
A crucial claim which appears in \cite{Lee} is the following: If $E^{*}$ is separable, then $(E_p)^{*}$ also is separable. However, this claim is wrong.
Therefore, the proof fails.  
\end{remark}

\newpage
\centerline{{\bf Statements and Declarations}}

The authors declare that no  support was received during the preparation of
this manuscript. The authors have no relevant financial or non-financial interests
to disclose.

\centerline{{\bf Conflict of interest}} 

There is no conflict of interest.

\centerline{{\bf Data availability}}

Data sharing not applicable to this article as no datasets were generated or
analyzed during the current study.
\vspace{0.3in}

\begin{thebibliography}{99}


\bibitem{Asplund} E. Asplund, 
\newblock \textit{Fr\'echet differentiability of convex functions}, Acta Math. \textbf{121} (1968), 31--47.

\bibitem{AS} V. I. Averbukh and O. G. Smolyanov, 
\newblock \textit{The theory of differentiation in linear topological spaces},
 Russian Math. Surveys, \textbf{22} (1967), 201--258.

\bibitem{CP} J. Chaber, R. Pol,
\newblock \textit{On hereditarily Baire spaces, $\sigma$-fragmentability of mappings and Namioka property}, 
Topology Appl. \textbf{151} (2005), 132--143.

\bibitem{Cheng_Fabian} L. Cheng and M. Fabian,
\newblock \textit{The product of a G\^ ateaux differentiability space and a separable space is a G\^ ateaux differentiability space},
Proc. Amer. Math. Soc. \textbf{129} (2001), 3539--3541.

\bibitem{Corbacho} E. Corbacho, A. Plichko, V. Tarieladze,
\newblock \textit{A one-sided version of Alexiewicz-Orlicz's differentiability theorem},
Rev. R. Acad. Cienc. Exactas F\'is. Nat. Ser. A Mat. RACSAM \textbf{99} (2005), 167–-181.

\bibitem{Coban} M. \v Coban and P. Kenderov,
\newblock \textit{Dense G\^ ateaux differentiability of the $sup$-norm in $C(T)$
and the topological properties of $T$},
 C.R. Acad. Bulgare Sci. \textbf{38} (1985), 1603--1604.

\bibitem{Fabian} M. Fabian,
\newblock \textit{G\^ ateaux Differentiability of Convex Functions and Topology. Weak Asplund Spaces},
 Canadian Math. Soc. Series of  Monographs and Advanced Texts,
John Wiley $\&$ Sons, 1997.

\bibitem{Fabian1} M. Fabian, P. Habala, P. H\'ajek, V. Montesinos, V. Zizler,
\newblock \textit{Banach space theory. The basis for linear and nonlinear analysis}, 
CMS Books Math./ Ouvrages Math. SMC, Springer, New York, 2011. 

\bibitem{DGZ} R. Deville, G. Godefroy, and V. Zizler,
\newblock \textit{Smoothness and renormings in Banach spaces},
 Pitman Monographs 64, Longman, London, 1993.

\bibitem{Eyland_Sharp1} R. Eyland and B. Sharp,
\newblock \textit{A factor theorem for locally convex differentiability spaces},
 Bull. Austral. Math. Soc. \textbf{43} (1991), 101--113.

\bibitem{Eyland_Sharp2} R. Eyland and B. Sharp,
\newblock \textit{Convex spaces: Classification by differentiable convex functions},
 Bull. Austral. Math. Soc. \textbf{46} (1992), 127--138.

\bibitem{Eyland_Sharp3} R. Eyland and B. Sharp,
\newblock \textit{The Fr\'echet differentiability of convex functions on $C(S)$},
Miniconference on probability and analysis, 
Proc. Centre Math. Appl. Austral. Nat. Univ. \textbf{29} (1992), 73--91.

\bibitem{Keller} H. H. Keller,
\newblock \textit{Differential calculus in locally convex spaces},
Lecture Notes in Math. 417, Springer-Verlag, Berlin, Heidelberg, New York, 1974.

\bibitem{Larman} D. G. Larman and R. R. Phelps,
\newblock \textit{G\^ ateaux differentiability of convex functions on Banach spaces},
J. London Math. Soc. \textbf{20} (1979), 115--127.

\bibitem{Mazur} S. Mazur,
\newblock \textit{\"Uber konvexe Mengen in linearen normierten R\"aumen},
 Studia Math. \textbf{4} (1933), 70--84.

\bibitem{Moors_Baire} W. B. Moors, 
\newblock \textit{The product of a Baire space with a hereditarily Baire metric space is Baire},
Proc. Amer. Math. Soc. \textbf{134} (2006), 2161--2163.

\bibitem{Moors} W. B. Moors and S. Somasundaram,
 \textit{A G\^ ateaux differentiability space that is not weak Asplund},
 Proc. Amer. Math. Soc. \textbf{134} (2006), 2745--2754.

\bibitem{Namioka} I. Namioka and R. R. Phelps, 
\newblock \textit{Banach spaces which are Asplund spaces}, Duke Math. J. \textbf{42} (1975), 735--750.

\bibitem{Phelps} R. R. Phelps,
\newblock \textit{Convex Functions, Monotone Operators and Differentiability},
Lecture Notes in Math. 1364, Springer-Verlag, Berlin, Heidelberg, New York, 1989.

\bibitem{Sharp} B. Sharp,
\newblock \textit{The differentiability of convex functions on topological linear spaces},
 Bull. Austral. Math. Soc. \textbf{42} (1990), 201--213.

\bibitem{Shen_Cheng} X. Shen and L. Cheng, 
\newblock \textit{On the product of G\^ ateaux differentiability locally convex spaces},
Acta Mathematica Sinica, Ser. B (Engl. Ed.) \textbf{25} (2005), 395--400.

\bibitem{Tal1} M. Talagrand,
\newblock \textit{Deux examples de fonctions convexes},
C. R. Acad.Sci. Paris \textbf{288} (1979), 461--464.

\bibitem{Valdivia} M. Valdivia,
\newblock \textit{Products of Baire topological vector spaces},
 Fund. Math. \textbf{125} (1985), 71--80.

\bibitem{Lee} C. Wu, X. Wang, L. Cheng and E. S. Lee,
\newblock \textit{On the Asplund property of locally convex spaces}, 
J. Math. Anal. Appl. \textbf{204} (1996), 432--443.

\bibitem{Yamamuro} S. Yamamuro,
\newblock \textit{Differential calculus in topological linear spaces},
Lecture Notes in Math. 374, Springer-Verlag, Berlin, Heidelberg, New York, 1974.

\bibitem{Zheng} Xi Yin Zheng, Kung Fu Ng,
\newblock \textit{Differentiability of convex functions on a locally convex topological vector space}, J. Convex Anal.  \textbf{26} (2019), 761--772.


\end{thebibliography}
\end{document}